\documentclass[oneside, 11pt]{amsart}

\usepackage[letterpaper,body={15.6cm,24cm}, mag=1000]{geometry}
\usepackage{amssymb}
\usepackage{amsthm}
\usepackage{amscd}

\numberwithin{equation}{section}
\theoremstyle{plain}
\newtheorem{cor}[equation]{Corollary}
\newtheorem{lemma}[equation]{Lemma}

\newtheorem{prop}[equation]{Proposition}

\newtheorem{thm}[equation]{Theorem}

\newtheorem*{thma}{Theorem A}
\newtheorem*{thmb}{Theorem B}

\theoremstyle{definition}

\newtheorem{remark}[equation]{Remark}

\newcommand{\dlabel}[1]{\ifmmode \text{\ttfamily \upshape [#1] } \else
{\ttfamily \upshape [#1] }\fi \label{#1}}

\newcommand{\Z}{\operatorname{Z} }

\renewcommand{\Pr}{\operatorname{Pr} }

\newcommand{\gen}[1]{\left < #1 \right >}
\newcommand{\Aut}{\operatorname{Aut} }

\begin{document}

\baselineskip 15pt

\title[A  given element is a commutator of two group elements]{Probability that a  given element of a group is a commutator of any  two randomly chosen group elements}

\author{Rajat K.~Nath and Manoj K.~Yadav}

\address{Deparment of Mathematical Sciences, Tezpur University, Tezpur\\ 
Sonitpur, Assam-784 028, INDIA}

\address{School of Mathematics, Harish-Chandra Research Institute \\
Chhatnag Road, Jhunsi, Allahabad - 211 019, INDIA}
\email{rajatkantinath@yahoo.com and myadav@hri.res.in}

\subjclass[2000]{Primary 20D60; Secondry 20P05}
\keywords{Camina group, conjugacy class size, isoclinism of groups}

\begin{abstract}
We study the probability of  a given element, in the commutator subgroup of a group, to be equal to a commutator of two randomly chosen group elements, and compute explicit formulas for calculating this probability for some interesting classes of groups having only two different conjugacy class sizes. We re-prove the fact that if $G$ is a finite group such that the set of its conjugacy class sizes is $\{1, p\}$, where $p$ is a prime integer, then $G$ is isoclinic (in the sense of P. Hall) to an extraspecial $p$-group.
\end{abstract}

\maketitle

\section{Introduction}

For a finite group $G$, let ${\Pr}(G)$ denote the commutativity degree or commuting probability of $G$, which is defined by
\[{\Pr}(G) = |\{(x, y) \in G \times G \mid xy = yx\}|/|G|^2,\] where, for any finite set $S$, $|S|$ denotes its cardinality. Various aspects of this notion have been studied by many mathematicians over the years. A very impressive historical account can be found in \cite[Introduction]{pH12}. Recently,  Pournaki and Sobhani \cite{PS08} introduced the notion of ${\Pr}_g(G)$, which is defined  by
\[{\Pr}_g(G) =  |\{(x, y) \in G \times G \mid x^{-1}y^{-1} xy = g\}|/{|G|^2}.\]
Notice that for a given element $g \in G$, ${\Pr}_g(G)$ measures the probability that the commutator of two randomly chosen group elements is equal to $g$. Obviously, when $g = 1$, ${\Pr}_g(G) = {\Pr}(G)$.  Pournaki and Sobhani \cite{PS08} mainly studied ${\Pr}_g(G)$ for finite groups $G$ which have only two different irreducible complex character degrees and obtained an impressive formula for ${\Pr}_g(G)$ for such groups $G$. In particular, using character theoretic techniques, they obtained explicit formulas for ${\Pr}_g(G)$, when $G$ is a finite group with  $|\gamma_2(G)| = p$, where $p$ is a prime integer and $\gamma_2(G)$ denotes the commutator subgroup of $G$.  In this situation, there  can only be two cases, namely, (i) $\gamma_2(G) \le \Z(G)$ or (ii) $\gamma_2(G) \cap \Z(G) = 1$, where $\Z(G)$ denotes the center of $G$. These are the cases which were studied by Rusin \cite{dR79} in 1979 and explicit formulas were ontained for ${\Pr}(G)$. 
For finite groups $G$ which have only two different irreducible complex character degrees $1$ and $m$ (say), they proved \cite[Theorem 2.2]{PS08} that for each $1 \neq g \in K(G)$, 
\begin{equation}\label{1eqn1}
\Pr_g(G) = (1/|\gamma_2(G)|)(1 - 1/m^2),
\end{equation}
where $K(G)$ denotes the set of all commutators in $G$. We assume that $1 \in K(G)$.

Motivated by the results of  Rusin \cite{dR79}, and Pournaki and Sobhani \cite{PS08}, we investigate ${\Pr}_g(G)$ for some classes of finite groups $G$ which have only two different conjugacy class sizes.  Explicit formulas are obtained for some interesting  classes of finite groups using purely group theoretic techniques.  A finite group $G$ is said to be a Camina group if $x^G = x\gamma_2(G)$ for all $x \in G - \gamma_2(G)$. It was proved by Dark and Scoppola \cite{DS96} that the nilpotency class of a finite Camina $p$-group is at most $3$. For such groups  $G$ of nilpotency  class $2$, we obtain a very nice formula for ${\Pr}_g(G)$ in the following theorem, proof of which follows from Theorem \ref{4thm1}. Explicit formulas for such groups of class $3$ are obtained in Section $5$.

\begin{thma} 
Let $G$ be a finite Camina $p$-group of nilpotency class $2$ such that $|\gamma_2(G)|  = p^r$, where $p$ is a prime number and $r$ is a positive integer. Then, for $g \in K(G)$, 
\[
{\Pr}_g(G) = \begin{cases} \frac{1}{p^r}\left(1 + \frac{p^r - 1}{p^{2m}}\right) \text{ if } g = 1\\
\frac{1}{p^r}\left(1 - \frac{1}{p^{2m}}\right) \text{ if } g \ne 1,
\end{cases}
\]
for some positive integer $m$ such that $r < m$. 
\end{thma}

Denote by $P(G)$ the set $\{\Pr_g(G)  \mid 1 \neq g \in K(G)\}$. It follows from \cite[Theorem 2.2]{PS08} (see \eqref{1eqn1}) that $P(G)$ is a singleton set for all  finite groups $G$ which have only two different irreducible complex character degrees.  It follows from Theorem A that  $P(G)$ is a singleton set for all  finite Camina $p$-groups $G$ of nilpotency class $2$. Notice that  finite Camina $p$-groups of class $2$ forms a subclass of finite $p$-groups having only two different conjugacy class sizes. Are there also other classes of finite $p$-groups $G$ of class $2$ having only two different conjugacy class sizes and $|P(G)|=1$?  The answer is affirmative as we show in the following theorem, which we prove in Section 4.

\begin{thmb}
For any positive integer $r \ge 1$, there exists a group $G$ of order $p^{(r+1)(r+2)/2}$ such that $|P(G)| = 1$. Moreover, $\Pr_g(G) = (p^2 -1)/p^{2r+1}$ for each $1 \neq g \in K(G)$, and if $r >1$, $G$ is not a Camina group.
\end{thmb} 



Let $G$ be  an arbitrary finite group  having only two different conjugacy class sizes. Ito \cite{nI53} proved that such a group $G$ is  a prime power order group (upto abelian direct factors). It was then proved by  Ishikawa \cite{kI02} that the nilpotency class of such a group $G$ is bounded above by $3$. It follows from Theorem 2.3 (see below) that for studying $\Pr_g(G)$ for  arbitrary 
finite groups  having only two different conjugacy class sizes, it is sufficient to study it for such $p$-groups only.  Thus by Ishikawa's result, it is sufficient to consider such $p$-groups of nilpotency class $2$ and $3$ only.  In Section 5, we show that there exists a finite $p$-group $G$ of class $3$ having only two different conjugacy class sizes and  $|P(G)| > 1$. But, to the best of our knowledge, no example of a finite $p$-group $G$ of class $2$ having only two different conjugacy class sizes and $|P(G)| \neq 1$ is known. This gives rise to the following natural question.

\vspace{.1in}
\noindent{\bf Question.} Is it true that $|P(G)| = 1$ for all finite $p$-groups $G$ of class $2$ having only two different conjugacy class sizes?
\vspace{.1in}

In Section 5,  we also study some bounds on  $\Pr_g(G)$ for a finite group $G$. K. Ishikawa \cite{kI99} proved that  if $G$ is any finite group having only two different conjugacy class sizes $1$ and $p$, where $p$ is a prime integer, then $G$ is isoclinic (see Section 2 for the definition) to  an extraspecial $p$-group. This makes use of a  very deep result of Vaughan-Lee \cite[Main Theorem]{vL76}, which was  proved using Lie theoretic arguments.   We re-prove this statement using elementary arguments on commuting probability.

\vspace{.1in}

In the light of above discussion, it is interesting to pose the following natural problem.
\vspace{.1in}

\noindent{\bf Problem.} Classify all finite groups $G$ such that  $|P(G)| = 1$. 
\vspace{.1in}

 This problem can also be viewed in the following manner.  Let $\alpha : G \times G \to \gamma_2(G)$ be a commutator map defined by $\alpha(g_1, g_2) = [g_1, g_2]$ for each pair $(g_1, g_2) \in G \times G$. Let, for $g \in K(G)$, $n_g(G)$ denote the set $\alpha^{-1}(g)$. So $n_g(G)$, in some sense, can be viewed as a fiber on $g \in K(G)$ in $G \times G$. Now suppose that $N(G) := \{|n_g(G)| \mid 1 \neq g \in K(G)\}$.  Notice that $\Pr_g(G) = |n_g(G)|/|G|^2$. Thus classifying group $G$ with $|P(G)| = 1$ is nothing but classifying groups with $|N(G)| = 1$. Such things are defined and studied in \cite{SS09} for some other purposes.

\section{Preliminary results}

In 1940, P. Hall \cite{pH40} introduced the following concept of isoclinism on the class of all groups.

Let $X$ be a finite group and $\bar{X} = X/Z(X)$. 
Then commutation in $X$ gives a well defined map
$a_{X} : \bar{X} \times \bar{X} \to \gamma_2(X)$ such that
$a_{X}(xZ(X), yZ(X)) = [x,y]$ for $(x,y) \in X \times X$.
Two finite groups $G$ and $H$ are called \emph{isoclinic} if 
there exists an  isomorphism $\alpha$ of the factor group
$\bar G = G/Z(G)$ onto $\bar{H} = H/Z(H)$, and an isomorphism $\beta$ of
the subgroup $\gamma_2(G)$ onto  $\gamma_2(H)$
such that the following diagram is commutative
\begin{equation}\label{eqn1}
 \begin{CD}
   \bar G \times \bar G  @>a_G>> \gamma_2(G)\\
   @V{\alpha\times\alpha}VV        @VV{\beta}V\\
   \bar H \times \bar H @>a_H>> \gamma_2(H).
  \end{CD}
\end{equation}
The resulting pair $(\alpha, \beta)$ is called an \emph{isoclinism} of $G$ 
onto $H$. Notice that isoclinism is an equivalence relation among finite 
groups.

The following result is from \cite{pH40}.
\begin{thm}\label{2thm}
Let $G$ be any group. Then there exists a group $H$ such that $G$ and $H$ are isoclinic and $\Z(H) \le \gamma_2(H)$.
\end{thm}

Pournaki and  Sobhani \cite{PS08} proved the following interesting result. Since the result is very important for our study, we provide a brief proof here.

\begin{thm}\label{2thm1}
Let $G$ and $H$ be two isoclinic finite groups with isoclinism  $(\alpha, \beta)$. Then  ${\Pr}_g(G) = \Pr_{\beta(g)}(H)$.
\end{thm}
\begin{proof}
 Since $(\alpha, \beta)$ is an isoclinism from $G$ onto $H$,  diagram \eqref{eqn1} commutes. Let $g \in \gamma_2(G)$ be the given element. Consider the sets $S_g = \{(g_1Z(G), g_2Z(G)) \in G/Z(G)  \times G/Z(G) : [g_1, g_2] = g\}$ and $S_{\beta(g)} = \{(h_1Z(H), h_2Z(H)) \in H/Z(H)  \times H/Z(H) : [h_1, h_2] = \beta(g)\}$. Since the above diagram  commutes, it follows that $|S_g| = |S_{\beta(g)}|$. Since the maps $a_{G}$ and $a_{H}$ are  well defined, we have $|\{(g_1, g_2) \in G \times G : [g_1, g_2] = g\}| = |Z(G)|^2 |S_g|$ and $|\{(h_1, h_2) \in H \times H : [h_1, h_2] = \beta(g)\}| = |Z(H)|^2 |S_{\beta(g)}|$. Notice that $|G : Z(G)| = |H : Z(H)|$.  Hence
\[
{\Pr}_g(G) = \frac{|S_g|}{|G : \Z(G)|^2} = \frac{|S_{\beta(g)}|}{|H : Z(H)|^2 }  = {\Pr}_{\beta(g)}(H).
\]
This completes the proof. 
\end{proof}

We remark that Theorem \ref{2thm1} was proved by P. Lescot \cite[Lemma 2.4]{pL95} for $g = 1$.

Now we prove the following interesting result.
\begin{prop}\label{3prop1}
Let $G$ be a finite nilpotent group such that $\gamma_2(G)$ is finite of prime power order of some prime integer $p$. Then $G$ is isoclinic to a finite $p$-group $H$ such that $\Z(H) \le \gamma_2(H)$.
\end{prop}
\begin{proof}
By Theorem \ref{2thm}  there exists a group $H$ isoclinic to $G$ such that $\Z(H) \le \gamma_2(H) \cong \gamma_2(G)$. Since $G$ is nilpotent, $H$ is nilpotent and therefore it can be written as a direct sum of its Sylow $p$-subgroups. We claim that $H$ is a $p$-group, where $p$ is the prime integer given in the statement. Suppose that $q$ is a prime integer not equal to $p$ and $q$ divides $|H|$.  Then $H$ has a Sylow $q$-subgroup $Q$ (say) such that $\gamma_2(Q)$ is a $q$-group contained in $\gamma_2(G)$.  But, by our supposition, $\gamma_2(G)$ is a $p$-group. Thus $\gamma_2(Q)$ must be trivial. This shows  that $Q$ is abelian and therefore $Q \le \Z(H) \le \gamma_2(H)$, which is a contradiction to the fact that $\gamma_2(H) \cong \gamma_2(G)$ is a $p$-group. Hence our claim is true and the proof is complete. \hfill $\Box$

\end{proof}

A  $p$-group $G$ is said to be \emph{special} if $\Z(G) = \gamma_2(G) = \Phi(G)$, where $\Phi(G)$ denotes the Frattini subgroup of $G$.
 A special $p$-group $G$ is called \emph{extraspecial} if the order of $\Z(G)$ is $p$. Notice that a $p$-group $G$  is extraspecial if $\Z(G) = \gamma_2(G)$ is of order $p$.

\begin{cor}\label{3cor1}
Let $G$ be a finite group such that $|\gamma_2(G)| = p$ and $\gamma_2(G) \le \Z(G)$, where $p$ is a prime integer. Then $G$ is isoclinic to an extraspecial $p$-group.
\end{cor}

The following theorem follows from \cite{nI53}, \cite{kI02} and the definition of isoclinism.

\begin{thm}\label{3thm2}
Let $G$ be a finite group having only two different conjugacy class sizes $1$ and $m$ (say). Then $G$ is isoclinic to a finite $p$-group of class at most $3$ for some prime integer $p$. Moreover, $m = p^r$ for some positive integer $r$.
\end{thm}

The following result is due to I. D. Macdonald  \cite[Corollary 2.4]{iM81}.

\begin{thm}\label{3thm3}
Any finite Camina $p$-group of nilpotency class $2$ is special.
\end{thm}


\section{Proof of Theorem A}

For a group $G$ and any element $x \in G$, $[x, G]$ denotes the set $\{[x, g] = x^{-1}g^{-1}xg \mid g \in G\}$.
We start with the following extremely useful expression for ${\Pr}_g(G)$ from \cite{DN10}. We mention a slightly modified proof here.

\begin{lemma}\label{4lemma1}
Let $G$ be a finite group and $g \in G$. Then
\[
 {\Pr}_g(G) = \frac{1}{|G|}{\sum_{g \in [x, G]}}\frac{1}{|x^G|}.
\]
\end{lemma}
\begin{proof}
Notice that $\{(x, y) \in G \times G : [x, y] = g\} = \underset{x \in G}{\cup}(\{x\} \times T_x)$, where $T_x = \{y \in G : [x, y] = g\}$. Further notice that, for any $x \in G$, the set $T_x$ is non-empty if and only if $xg \in x^G$.
Suppose that $T_x$ is non-empty for some $x \in G$. Fix an element $t \in T_x$. It is easy to see that $T_x = tC_G(x)$. The proof of the lemma now follows from the definition of ${\Pr}_g(G)$.  \hfill $\Box$

\end{proof}

For a finite group $G$, by $b(G)$ we denote the size of the largest conjugacy class in $G$. Notice that $b(G) \le |\gamma_2(G)|$.
For  $g=1$,  the preceding lemma gives the following result.
\begin{prop}\label{4prop1}
Let $G$ be  a  finite group. Then
\[
\Pr(G) \geq \frac{1}{|b(G)|}\left(1 + \frac{|b(G)| - 1}{|G:Z(G)|}\right).
\]
 Moreover,\\
(i)  Equality holds if and only if $b(G) = |x^G|$ for all $x \in G - \Z(G)$;\\
(ii) If $G$ is non-abelian, then $\Pr(G) > \frac{1}{|b(G)|}$.
\end{prop}

\begin{thm}\label{4thm1}
Let $G$ be a finite nilpotent group of class $2$ such that $|\gamma_2(G)| = p^r$ and $[x, G] = \gamma_2(G)$ for all $x \in G - \Z(G)$, where $p$ is  a prime  and $r$ is a positive integer. Then 
\[
{\Pr}_g(G) = \begin{cases} \frac{1}{p^r}\left(1 + \frac{p^r - 1}{p^{2m}}\right) \text{ if } g = 1\\
\frac{1}{p^r}\left(1 - \frac{1}{p^{2m}}\right) \text{ if } g \ne 1,
\end{cases}
\]
for some positive integer $m$. Moreover, $m$ is independent of the choice of $G$ in its isoclinism family.
\end{thm}
\begin{proof}
Let $G$ be  a group  as given in the statement.  Then it follows  from Proposition  \ref{3prop1} that $G$ is isoclinic to a Camina  $p$-group $H$ (say) of class $2$. Suppose that $(\alpha, \beta)$ is an isoclinism between $G$ and $H$. It follows from \cite[Theorem 3.2]{iM81} that  $|H/\Z(H)| = p^{2m}$ for some positive integer $m$. Now by Theorem \ref{2thm1}, it follows that ${\Pr}_g(G) = {\Pr}_{\beta(g)}(H)$. So, now onwards we work with $H$ only.
Notice that $x^H = x\gamma_2(H)$ for all $x \in H - \Z(H)$.  Thus $G$ has only two conjugacy class sizes, namely, $1$ and $|\gamma_2(G)|$ and therefore   $b(H) = |\gamma_2(H)| = p^r$ for all $x \in H - \Z(H)$.  If $h = 1$, then  it follows from  Proposition \ref{4prop1}   that 
\[
{\Pr}_h(H) =  \frac{1}{p^r}\left(1 + \frac{p^r - 1}{|H : \Z(H)|}\right). 
\]

Now assume that $h \ne 1$.   Notice that if $x \in \Z(H)$ and $h \in [x, H]$, then $h = 1$.  Since $h \in [x, H]$ and $|x^H| = b(H)$ for all $x \in H - \Z(H)$,  by  Lemma \ref{4lemma1} we have 
\begin{align*}
{\Pr}_h(H) & = \frac{1}{|H|}\sum_{x \in H - \Z(H)}\frac{1}{|x^H|} = \frac{|H - \Z(H)|}{|b(H)||H|}\\
& = \frac{1}{b(G)}\left(1 - \frac{1}{|H : \Z(H)|}\right). 
\end{align*}

Since $|H:\Z(H)|  = p^{2m}$, we have
\[
{\Pr}_h(H) = \begin{cases} \frac{1}{p^r}\left(1 + \frac{p^r - 1}{p^{2m}}\right) \text{ if } h = 1\\
\frac{1}{p^r}\left(1 - \frac{1}{p^{2m}}\right) \text{ if } h \ne 1.
\end{cases}
\] 
Since $\beta$ is an isomorphism from $\gamma_2(G)$ onto $\gamma_2(H)$, for each $g \in \gamma_2(G)$, there exists an $h \in \gamma_2(H)$ such that $h = \beta(g)$. The fact that $|G_1/\Z(G_1)| = |G_2/\Z(G_2)|$ for any two isoclinic finite groups $G_1$ and $G_2$ shows that the integer $m$ appeared above is independent of the choice of $G$ in its isoclinism family.  This   completes the proof of the theorem.   \hfill $\Box$

\end{proof}

Let $G$ be a  finite group and $H$ be it's normal subgroup of order $p$, where $p$ is a prime dividing $|G|$ such that $(|G|, p-1) = 1$. Then it follows from  $N/C$ lemma (i.e., $N_G(H)/C_G(H)$ embeds in $\Aut(H)$) that $H$ is a central subgroup of $G$. As a corollary of the preceding theorem, we  get the following interesting result.
\begin{cor}
Let $G$ be a  finite group with $|\gamma_2(G)| = p$, where $p$ is a  prime integer.  If either $G$ is nilpotent or $p$ is a prime dividing $|G|$ such that $(|G|, p-1) = 1$,  then \\
(i)  $G$ is isoclinic to an extraspecial $p$-group;\\
(ii) for $g \in \gamma_2(G)$, we have
\[
{\Pr}_g(G) = \begin{cases} \frac{1}{p}\left(1 + \frac{p - 1}{p^{2m}}\right) \text{ if } g = 1\\
\frac{1}{p}\left(1 - \frac{1}{p^{2m}}\right) \text{ if } g \ne 1,
\end{cases}
\] for some positive integer $m$.
\end{cor}
\begin{proof}
Notice that  $\gamma_2(G)$ is a central subgroup of $G$ in both of the cases. Thus $G$ is a nilpotent group of class $2$ and therefore,  by Corollary \ref{3cor1}, it is isoclinic to an extraspecial $p$-group. Now  the result follows from  Theorem \ref{4thm1} for $r = 1$. \hfill $\Box$

\end{proof}

In the following result we characterize (upto isoclinism) finite groups $G$ such that $\Pr(G) = \frac{1}{|b(G)|}\left(1 + \frac{|b(G)| - 1}{|G:Z(G)|}\right)$. 
\begin{prop}
Let $G$ be a non-abelian finite group such that $\Pr(G) = \frac{1}{|b(G)|}\left(1 + \frac{|b(G)| - 1}{|G:Z(G)|}\right)$. Then $G$ is isoclinic to a finite $p$-group having only two different conjugacy class sizes $1$ and $p^r$ for some prime $p$ and positive integer $r$. Moreover, if $b(G) = |\gamma_2(G)|$, then  $G$ is isoclinic to  a Camina special $p$-group.
\end{prop}
\begin{proof}
Let $G$ be the group as in the statement. Then it follows from Proposition \ref{4prop1}(i) that $G$ has only two conjugacy class sizes, namely, $1$ and $b(G)$. Now it follows from Theorem \ref{3thm2}  that $G$ is isoclinic to a finite  $p$-group of nilpotency class  at most $3$ for some prime integer $p$ and $b(G) = p^r$ for some positive integer $r$. Thus the first assertion of the statement holds true.  

Now assume that $b(G) = |\gamma_2(G)|$.  We claim that  the nilpotency class of $G$ is at most $2$. Supose that the nilpotency class is $3$. Then there exists an element $u \in \gamma_2(G) - Z(G)$. Then $1 \ne |u^G| = |[u, G]| \le |\gamma_3(G)| <  |\gamma_2(G)| = b(G)$, which is a contradiction to the fact that $G$ has only two conjugacy class sizes.  Since $G$ is non-abelian, it now follows that the nilpotency class of $G$ is $2$.  As (by Theorem \ref{2thm}) there exists a finite $p$-group $H$ such that $G$ and $H$ are isoclinic and $\Z(H) \le \gamma_2(H)$,  it follows that $H$ is a Camina group of nilpotency class $2$. That $H$ is special, now follows from Theorem \ref{3thm3}. This completes the proof. \hfill $\Box$

\end{proof}

\section{Proof of Theorm B}

For any positive integer $r \ge 1$, consider the  following group constructed by Ito \cite{nI53}.
\begin{eqnarray}\label{6eqn1}
G  &= \langle x_1, \ldots, x_{r+1} \mid [x_i, x_j] = y_{ij}, [x_k, y_{ij}] = 1,\\
 & x_i^p = x_{r+1}^p = y_{ij}^p, 1 \le i < j \le r+1\rangle.\nonumber
\end{eqnarray}

Some interesting properties of this group are given in the following lemma,  proof of which follows from \cite[Example 1]{nI53}.
\begin{lemma}\label{6lemma0}
The group $G$ defined in \eqref{6eqn1} is a special $p$-group of order $p^{(r+1)(r+2)/2}$ and exponent $p$,  and $|\gamma_2(G)|  = p^{r(r+1)/2}$. This group has only two different conjugacy class sizes, namely $1$ and $p^r$.
\end{lemma}

We now proceed to prove Theorem B.  We start with the following technical result, which is also of independent interest.

\begin{lemma}\label{6lemma1}
Let $G$ be a finite $p$-group of nilpotency class $2$ minimally generated by $x_1, \ldots, x_d$ such that the exponent of $\gamma_2(G)$ is $p$. Let $g_1 = \Pi_{i=1}^d x_i^{\alpha_i}$ and $g_2 = \Pi_{j=1}^d x_j^{\beta_j}$ be two different elements of $G$ such that  $[g_1, g_2] \neq 1$, where  $\alpha_i$, $\beta_j$ are some non negative integers between $0$ and $p-1$. Then we can find another minimal generating set for $G$ containing $g_1, g_2$.
\end{lemma}
\begin{proof}
Let $k$ and $l$ be the smallest integers such that $\alpha_k \neq 0$ and $\beta_l \neq 0$ respectively.  First suppose that  $k \neq l$. Assume that  $k < l$. Then it is not difficult to see that the set
\[\{x_1, \ldots, x_{k-1}, g_1, x_{k+1}, \ldots, x_{l-1}, g_2,  x_{l+1}, \ldots, x_d\}\]
 minimally generates $G$. Now supose that $k = l$. If $\alpha_l = \beta_k$, then choose the next smallest $m$ such that $\alpha_m \neq \beta_m$. Let $t$ denote the largest integer $<m$ such that $\alpha_t = \beta_t$. Then it follows that if we replace $x_t$ by $g_1$ and $x_m$ by $g_2$ in the given generating set, we'll again get a minimal generating set. If $\alpha_l \neq \beta_k$, then the following three cases may occur:

(i) There exists an integer $s$, $k \le s \le d$, which is the smallest  such that  $\alpha_s = \beta_s \neq 0$.

(ii) There exists an integer $s$, $k \le s \le d$, which is the smallest such that $\alpha_s = 0$, $\beta_s \neq 0$ or $\alpha_s \neq 0$, $\beta_s = 0$.

(iii) For each $k \le s \le d$, $\alpha_s \neq 0$ if and only if  $\beta_s \neq 0$, and $\alpha_s \neq \beta_s$ if these are different from zero.

In the situation of case (i), we can get a required minimal generating set by replacing $x_{s-1}$ by $g_1$ and $x_s$ by $g_2$ in the given generating set. If case (ii) occurs, then choose the lagest $t < s$ such that $\alpha_t \neq 0$, $\beta_t \neq 0$. So by replacing $x_{t}$ by $g_1$ and $x_s$ by $g_2$ in the given generating set, we get a required minimal generating set.

Now assume the last case, i.e.,  for each $1 \le s \le d$, $\alpha_s \neq 0$ if and only if  $\beta_s \neq 0$, and $\alpha_s \neq \beta_s$ if these are different from zero. Let us re-write $g_1$ and $g_2$ by ommiting $x_i$'s for which $\alpha_i = \beta_i = 0$. So $g_1 = \Pi x_{e_i}^{\alpha_{e_i}}$ and $g_2 = \Pi x_{f_j}^{\beta_{f_j}}$, where $e_i$ and $f_j$ occur in the increasing order. In the rest of the proof, we confine all of our computations in ${\mathbb F}_p$, the field consisting of $p$ elements.

Suppose that there exists some positive integer $t$ such that $\alpha_{e_t}\beta_{e_{t+1}} \neq  \alpha_{e_{t+1}}\beta_{e_{t}}$.  Consider the set 
\[X := \{x_1, \ldots, x_{e_t-1}, g_1, x_{e_t+1}, \ldots, x_{e_{t+1}-1}, g_2,  x_{e_{t+1}+1}, \ldots, x_d\}.\]
 It is easy to see that $y_1 := x_{e_t}^{\alpha_{e_t}}x_{e_{t+1}}^{\alpha_{e_{t+1}}}$ and  $y_2 := x_{e_t}^{\beta_{e_t}}x_{e_{t+1}}^{\beta_{e_{t+1}}}$ can be produced by elementary cancelations in the set $X$. At this stage, we can have two possibilities, namely (a) $[x_{e_t}, x_{e_{t+1}}] \neq 1$ or (b) $[x_{e_t}, x_{e_{t+1}}] = 1$. First assume (a), i.e., $[x_{e_t}, x_{e_{t+1}}] \neq 1$.
Since $\alpha_{e_t}\beta_{e_{t+1}} \neq  \alpha_{e_{t+1}}\beta_{e_{t}}$, we claim that the elements $y_1$ and $y_2$ can not commute with each other. Indeed, if $[y_1, y_2] = 1$, then by a straight forwrd calculation it follows that 
\[[x_{e_t}, x_{e_{t+1}}]^{\alpha_{e_t}\beta_{e_{t+1}} - \alpha_{e_{t+1}}\beta_{e_{t}}} =1.\]
Since the exponent of $\gamma_2(G)$ is $p$, this is possible only when $\alpha_{e_t}\beta_{e_{t+1}} = \alpha_{e_{t+1}}\beta_{e_{t}}$, which contradicts our supposition.  Hence our claim follows.
Consider the subgroup $H = \gen{x_{e_t}, x_{e_{t+1}}}$. Since the nilpotency class of $G$ is $2$ and $[x_{e_t}, x_{e_{t+1}}] \neq 1$,  $H$ is a non-abelian group of class $2$. Obviously $y_1, y_2 \in H - \Phi(H)$, where $\Phi(H)$ denotes the Frattini subgroup of $H$. Since $[y_1, y_2] \neq 1$, it follows that $y_1$ and $y_2$ generate $H$. This proves that $X$ generates $G$ and therefore $X$ is a minimal generating set for $G$.  

Now assume (b), i.e., $[x_{e_t}, x_{e_{t+1}}] = 1$. Let $\gamma_{e_t}$ and $\delta_{e_t}$ be the multiplicative inverses of $\alpha_{e_t}$ and $\beta_{e_t}$ in ${\mathbb F}_p$ respectively. Notice that 
\[y_1^{\gamma_{e_t}}y_2^{-\delta_{e_t}} = x_{e_{t+1}}^{\alpha_{e_{t+1}}\gamma_{e_t} - \beta_{e_{t+1}}\delta_{e_t}}.\]
We claim that $\alpha_{e_{t+1}}\gamma_{e_t} - \beta_{e_{t+1}}\delta_{e_t}$ not equal to zero in ${\mathbb F}_p$. If possible, assume the contrary. Thus $\alpha_{e_{t+1}}\gamma_{e_t} = \beta_{e_{t+1}}\delta_{e_t}$. Multiplying both sides by $\alpha_{e_t}\beta_{e_t}$, we get 
$\alpha_{e_t}\beta_{e_{t+1}} =  \alpha_{e_{t+1}}\beta_{e_{t}}$, since $\alpha_{e_t} \gamma_{e_t} = 1 =\beta_{e_t} \delta_{e_t}$. This gives a contradiction to our supposition. Hence our claim is true. Now it is easy to see that $y_1$ and $y_2$ generate the abelian subgroup $\gen{x_{e_t}, x_{e_{t+1}}}$. Hence $X$ is a minimal generating set for $G$.  

Finally assume that $\alpha_{e_t}\beta_{e_{t+1}} =  \alpha_{e_{t+1}}\beta_{e_{t}}$ for all possible positive integers $t$. In this case, we are going to show that $[g_1, g_2] = 1$, which is not possible by the given hypothesis and therefore this case does not occur.   Let $t'$ and $t''$ be two arbitrary integers such that $1 \le t' < t'' \le d'$.  Then there exist  integers $t_1, \ldots, t_r$ (say) such that $t' = t_1 < \cdots < t_r = t''$.  We now have  the following system of equations.
\[\alpha_{e_{t_i}}\beta_{e_{t_{i+1}}} = \alpha_{e_{t_{i+1}}}\beta_{e_{t_i}}, ~~~~~~1 \le i \le r-1.\]
Solving this system of equations, we get  $\alpha_{e_{t_1}}\beta_{e_{t_r}} = \alpha_{e_{t_r}}\beta_{e_{t_1}}$. This shows that 
\[[x_{e_{t'}}^{\alpha_{e_{t'}}}, \; x_{e_{t''}}^{\beta_{e_{t''}}}] [x_{e_{t''}}^{\alpha_{e_{t''}}}, \; x_{e_{t'}}^{\beta_{e_{t'}}}] = [x_{e_{t'}}, x_{e_{t''}}]^{\alpha_{e_{t'}}\beta_{e_{t''}} - \alpha_{e_{t''}}\beta_{e_{t'}}} = 1.\]
Since $t'$ and $t''$ were arbitrary, it follows that $[g_1, g_2] = 1$. This completes the proof of the lemma.
 \hfill $\Box$

\end{proof}

\begin{lemma}\label{6lemma2}
Let $G$ be the $p$-group defined in \eqref{6eqn1}, where $p$ is an odd prime. Then $|\{x \in G \mid y_{ij} \in [x, G]\}| = (p^2-1)p^{r(r+1)/2}$.
\end{lemma}
\begin{proof}
First we claim  that $|\{x \in G \mid y_{ij} \in [x, G]\}|$ is at least $(p^2-1)p^{r(r+1)/2}$. Consider the subgroup $Y_{ij}$ of $G$ generated by $x_i, x_j$. Notice that $Y_{ij}$ is a group of order $p^3$ and exponent $p$. So it follows that $y_{ij} \in [w, Y_{ij}]$ for all $w \in Y_{ij} - \gen{y_{ij}}$. Since the order of $\gamma_2(G) = \Z(G)$ is $p^{r(r+1)/2}$, it follows that  $|\{w \in Y_{ij}\gamma_2(G)  \mid y_{ij} \in [w, Y_{ij}\gamma_2(G)]\}| = (p^2-1)p^{r(r+1)/2}$. This proves our claim.

Now we show that $|\{x \in G \mid y_{ij} \in [x, G]\}|$ is not more than  $(p^2-1)p^{r(r+1)/2}$. Suppose that there exists a pair of elements $(g_1, g_2)$ in $G$ such that $y_{ij} = [g_1, g_2]$ and at least one of $g_1, g_2$ lie(s) outside $Y_{ij}\gamma_2(G)$. Let $g_1 = x_i^{\alpha_i} x_j^{\alpha_j} \Pi_{k \neq i, j} x_k^{\alpha_k}u_1$ and $g_2 = x_i^{\beta_i}x_j^{\beta_j} \Pi_{l \neq i, j}x_l^{\alpha_l}u_2$ for some non-negative integers $\alpha_i$, $\alpha_j$, $\beta_i$, $\beta_j$, $\alpha_k$, $\beta_l$, and some $u_1, u_2 \in \gamma_2(G)$.  Suppose that $g_1$ lies outside $Y_{ij}\gamma_2(G)$, and $x_k^{\alpha_k}$ and $x_k^{\beta_k}$ appears in $g_1$ and $g_2$ respectively with $\alpha_k \neq 0$. First suppose that $\beta_k = 0$. Consider the commutator $[x_i^{\alpha_i} x_j^{\alpha_j}x_k^{\alpha_k}, x_i^{\beta_i}x_j^{\beta_j}]$, which after expanding becomes $[x_i, x_j]^{\alpha_i\beta_j - \alpha_j\beta_i}[x_k, x_i]^{\alpha_k\beta_i}[x_k, x_j]^{\alpha_k\beta_j}$. Notice that the preceding expression is a part of $[g_1, g_2] = y_{ij} = [x_i, x_j]$. Since $\gamma_2(G)$ is independently generated by all $y_{rs}$, comparing powers of corresponding $y_{rs}$ we get
\[\alpha_i\beta_j - \alpha_j\beta_i = 1, ~~~~~~~\alpha_k\beta_i = 0, ~~~~~~~\alpha_k\beta_j = 0.\]
Since $\alpha_k \neq 0$ and we are in a field, it follows that both $\beta_i$ and $\beta_j$ are $0$. Hence $0 = \alpha_i\beta_j - \alpha_j\beta_i = 1$, which is not possible. This shows that $\beta_k$ can not be $0$ whenever $\alpha_k \neq 0$.  

Now assume that both $\alpha_k$ as well as $\beta_k$ are non zero. It is not difficult to  get similar kind of contradiction by considering the commutator $[x_i^{\alpha_i} x_j^{\alpha_j}x_k^{\alpha_k}, x_i^{\beta_i}x_j^{\beta_j}x_k^{\beta_k}]$ and comparing corresponding powers of $y_{rs}$. This shows that no pair $(g_1, g_2)$ in $G$ lying outside $Y_{ij}\gamma_2(G)$ can give $y_{ij} = [g_1, g_2]$. This completes the proof of the lemma. \hfill $\Box$

\end{proof}

Now we are ready to prove Theorem B.

\vspace{.1in}

\noindent {\it Proof of Theorem B.} Let $G$ be the group defined in \eqref{6eqn1}. Let $1 \neq g \in K(G)$. Since the nilpotency class of $G$ is $2$ and the exponent of $G$ is $p$ (notice that we only need the fact that the exponent of $G/\Z(G)$ is $p$),  there exist $g_1, g_2 \in G$ such that  $g_1 = \Pi_{i=1}^d x_i^{\alpha_i}$,  $g_2 = \Pi_{j=1}^d x_j^{\beta_j}$ and $1 \neq g = [g_1, g_2]$,  where  $\alpha_i$ and $\beta_j$ are some non negative integers between $0$ and $p-1$. By Lemma \ref{6lemma1}, we can  find a minimal generating set $\{w_1, \ldots, w_d\}$ for $G$ which includes both $g_1$ and $g_2$. Notice that  $w_i^p = 1$ for $1 \le i \le d$ and $\gamma_2(G)$ is minimally generated by  $[w_i, w_j]$, where $[w_i, w_j]$ are central elements of order $p$ for $1 \le i < j \le d$. Since $G$ has only two different conjugacy class sizes $1$ and $p^r$ (Lemma \ref{6lemma0}), we have $|x^G| = b(G) = p^r$ for all $x \in G - \Z(G)$. Thus by Lemma \ref{6lemma2}, it follows that $|\{x \in G \mid g \in [x, G]\}| = (p^2-1)p^{r(r+1)/2}$. Hence 
\[\Pr_g(G) = \frac{1}{|G|p^r}(p^2-1)p^{r(r+1)/2}= \frac{p^2 - 1}{p^{2r+1}},\]
since $|G| = p^{(r+1)(r+2)/2}$. Obviously $\Pr_g(G)$ is independent of $g$ and therefore $|P(G)| = 1$.

Notice that for $r = 1$, $G$ is a non-abelian group of order $p^3$. Thus it is an extraspecial $p$-group and therefore a Camina group. But if $r \ge 2$, then $|\gamma_2(G)| =  p^{r(r+1)/2} > p^r = b(G)$. This shows that $G$ is not a Camina group, because in a finite Camina $G$, $|\gamma_2(G)| = b(G)$. This  completes the proof. \hfill $\Box$

\vspace{.2in}

The following result follows  from \cite[Corollary 2.2]{kI02}.
\begin{lemma}\label{6lemma3}
Let $G$ be a finite $p$-group of nilpotency class $2$ having only two different conjugacy class sizes. Then both $G/\Z(G)$ as well as $\gamma_2(G)$ are elementary groups.
\end{lemma}

In the following result, we show that $P(H) = 1$ for all $p$-groups $H$  of class $2$ having only two different comjugacy class sizes and having the same commutator  structure as of  the group $G$ defined in \eqref{6eqn1} .
 
\begin{thm}\label{6prop4}
Let  $H$ be a finite $p$-groups of nilpotency class $2$ having only two different conjugacy class sizes. Further, let $H$ be minimally generated by  $\{w_1, w_2, \ldots, w_d\}$ such that $|\gamma_2(H)| = p^{d(d-1)/2}$.  Then $P(H) = 1$.
\end{thm}
\begin{proof}
Notice that $\Z(H) = \Phi(H)$. For, it follows from the preceding lemma that the exponent of $H/\Z(H)$ is $p$. Thus $H^p \le \Z(H)$.  Since $H$ is minimally generated by $d$ elements and $|\gamma_2(H)| = p^{d(d-1)/2}$, no central element can lie in $H - \Phi(H)$. Now using the fact that $\gamma_2(H) \le \Z(H)$, it follows that  $\Z(H) = \Phi(H)$.  By the preceding lemma it follows that $\gamma_2(H)$ is elementary abelian $p$-groups.   Let $G$ be the group defined in \eqref{6eqn1} with $r = d-1$. Notice that $G/\Z(G)$ as well as $\gamma_2(G)$ are elementary abelian and $\Z(G) = \Phi(G)$ (this may be observed directly from the presentation of the group or by using Lemma \ref{6lemma3} and the above information as $G$ satisfies the conditions of  Lemma \ref{6lemma3}). By the given hypothesis $|G/\Z(G)| = |H/\Z(H)|$ and $|\gamma_2(G)| =  |\gamma_2(H)|$.  Thus  $G/\Z(G) \cong H/\Z(H)$ and $\gamma_2(G) \cong  \gamma_2(H)$. Since   $|\gamma_2(H)| = p^{d(d-1)/2}$, $[w_i, w_j] \neq 1$ for all $1 \le i < j \le d$. Set  $[w_i, w_j] = z_{ij}$. Notice that  the map $\alpha : G/\Z(G) \to H/\Z(H)$ defined on the set of generators by $\alpha(x_i) = w_i$ gives an isomorphism of   $G/\Z(G)$ onto $H/\Z(H)$. Similarly the map  $\beta : \gamma_2(G) \to \gamma_2(H)$ defined on the set of generators by $\beta(y_{ij}) = z_{ij}$ 
gives an isomorphism of   $\gamma_2(G)$ onto $\gamma_2(H)$.  It is not difficult to show that diagram \eqref{eqn1} commutes in the present setup. Thus it follows that $G$ and $H$ are isoclinic. Hence by Theorem \ref{2thm1} and Theorem B we have  $P(H) = P(G) = 1$.      \hfill $\Box$

\end{proof}


\section{Some more examples and  bounds for ${\Pr}_g(G)$}

As promised in the introduction, we now show the existence of a finite group $G$ of class $3$ such that  $G$ has only two different conjugacy class sizes and $|P(G)| > 1$. Consider the following group for an odd prime $p$.
\begin{equation}\label{70eqn1}
G = \gen{x_1, x_2 \mid [x_1, x_2] = y, [x_1, y] = z_1, [x_2, y] = z_2, x_i^p = y^p = z_i^p = 1 (i = 1, 2)}.
\end{equation}
That $G$ has only two conjugacy class sizes $1$ and $p^2$,  follows from \cite[Theorem 4.2]{kI99}. It is easy to see that the nilpotency class of $G$ is three, $|\gamma_2(G)| = p^3$,  $|\Z(G)| = p^2$ and $|G| = p^5$. 

Let $g_1$ and $g_2$ be two elements of $G$ modulo $\Z(G)$. Then $g_1 = x_1^{\alpha_1}x_2^{\beta_1}y^{\gamma_1}$ and $g_2 = x_1^{\alpha_2}x_2^{\beta_2}y^{\gamma_2}$ for $0 \le \alpha_i, \beta_i, \gamma_i \le p-1$, where $i = 1, 2$. We are now going to calculate $[g_1, g_2]$. 
\begin{eqnarray}\label{70eqn2}
[g_1, g_2] &=& [x_1^{\alpha_1}x_2^{\beta_1}y^{\gamma_1}, x_1^{\alpha_2}x_2^{\beta_2}y^{\gamma_2}] = [x_1^{\alpha_1}x_2^{\beta_1}, x_1^{\alpha_2}x_2^{\beta_2}] z_1^{\alpha_1\gamma_2-\alpha_2\gamma_1} z_2^{\beta_1\gamma_2 - \beta_2\gamma_1}\\
&=& [x_2^{\beta_1}, x_1^{\alpha_2}][[x_2^{\beta_1}, x_1^{\alpha_2}], x_2^{\beta_2}]   [x_1^{\alpha_1}, x_2^{\beta_2}][[x_1^{\alpha_1}, x_2^{\beta_2}], x_2^{\beta_1}] z_1^{\alpha_1\gamma_2-\alpha_2\gamma_1} z_2^{\beta_1\gamma_2 - \beta_2\gamma_1}.\nonumber
\end{eqnarray}

It is not difficult to show that
\begin{eqnarray*}
[x_1^{\alpha_2}, x_2^{\beta_1}] =  y^{\alpha_2\beta_1} z_1^{-\beta_1\alpha_2(\alpha_2-1)/2} z_2^{-\alpha_2\beta_1(\beta_1-1)/2}.
\end{eqnarray*}
and 
\begin{eqnarray*}
[x_1^{\alpha_1}, x_2^{\beta_2}] =  y^{\alpha_1\beta_2} z_1^{-\beta_2\alpha_1(\alpha_1-1)/2} z_2^{-\alpha_1\beta_2(\beta_2-1)/2}.
\end{eqnarray*}

Putting these values in \eqref{70eqn2}, we get
\begin{eqnarray}\label{70eqn3}
[g_1, g_2] &=& y^{\alpha_1\beta_2 - \alpha_2 \beta_1} z_1^{ \beta_1\alpha_2(\alpha_2-1)/2 - \beta_2\alpha_1(\alpha_1-1)/2 + \alpha_1\gamma_2-\alpha_2\gamma_1}\\
& &  z_2^{\alpha_2\beta_1(\beta_1-1)/2 - \alpha_1\beta_2(\beta_2-1)/2 - \alpha_1\beta_1\beta_2 + \alpha_2\beta_1\beta_2 + \beta_1\gamma_2 - \beta_2\gamma_1}.\nonumber
\end{eqnarray}

\begin{lemma}
Let $G$ be the group as defined in \eqref{70eqn1}. Then $\Pr_y(G) = \frac{p^2-1}{p^4}$
\end{lemma}
\begin{proof}
Let $g_1 = x_1^{\alpha_1}x_2^{\beta_1}y^{\gamma_1}$ and $g_2 = x_1^{\alpha_2}x_2^{\beta_2}y^{\gamma_2}$ be two arbitray elements of $G$ modulo the center such that $[g_1, g_2] = y = [x_1, x_2]$.  Using \eqref{70eqn3} and comparing powers of $y$, $z_1$ and $z_2$, we get the following system of equations:
\begin{eqnarray}
\label{70eqn4}&&\alpha_1\beta_2 - \alpha_2 \beta_1 = 1\\
\label{70eqn5}&&\beta_1\alpha_2(\alpha_2-1)/2 - \beta_2\alpha_1(\alpha_1-1)/2 + \alpha_1\gamma_2-\alpha_2\gamma_1 = 0\\
\label{70eqn6}&&\alpha_2\beta_1(\beta_1-1)/2 - \alpha_1\beta_2(\beta_2-1)/2 - \alpha_1\beta_1\beta_2 + \alpha_2\beta_1\beta_2 + \beta_1\gamma_2 - \beta_2\gamma_1 = 0
\end{eqnarray}

It follows from \eqref{70eqn4} that both of $\alpha_1$ and $\beta_1$ can not be zero.
First assume that none of $\alpha_1, \beta_1$ is zero.  Then substituting  the value of $\beta_2$ from \eqref{70eqn4} in \eqref{70eqn5} and \eqref{70eqn6}, and cancelling out $\gamma_2$ using the two new equations, we get
\[\alpha_2 = \frac{\beta_1^2\alpha_1 + \beta_1\alpha_1 - \beta_1 + \alpha_1 - 2\gamma_1 -1}{\beta_1^2}.\]
Now we can find $\beta_2$ and $\gamma_2$ using \eqref{70eqn4} and  \eqref{70eqn5}.

Now assume that  $\alpha_1 = 0$, $\beta_1 \neq 0$  (or $\beta_1 = 0$, $\alpha_1 \neq 0$). Then, using \eqref{70eqn4}-\eqref{70eqn6}, it is not difficult to find $\alpha_2, \beta_2$ and $\gamma_2$ in terms of $\beta_1$ ( or $\alpha_1$) and $\gamma_1$.  Thus it follows that given any element $g_1 = x_1^{\alpha_1}x_2^{\beta_1}y^{\gamma_1}z \in G - \gamma_2(G)$, where $z \in \Z(G)$, there exists an element $g_2 = x_1^{\alpha_2}x_2^{\beta_2}y^{\gamma_2} \in G$ such that $[g_1, g_2] = y$. Thus $|\{g \in G \mid y \in [g, G]\}| = p^5 - p^3$. Hence by Lemma \ref{4lemma1}, $\Pr_{y}(G) = \frac{1}{|G|b(G)}(p^5 - p^3) = \frac{p^2-1}{p^4}$, which is the required value for $\Pr_{y}(G)$. \hfill $\Box$

\end{proof}

\begin{lemma}
Let $G$ be the group as defined in \eqref{70eqn1}. Then $\Pr_{z_1}(G)  = \frac{p^2-1}{p^5}$.
\end{lemma}
\begin{proof}
Let $g_1 = x_1^{\alpha_1}x_2^{\beta_1}y^{\gamma_1}$ and $g_2 = x_1^{\alpha_2}x_2^{\beta_2}y^{\gamma_2}$ be two arbitray non-trivial elements of $G$ modulo the center such that $[g_1, g_2] = z_1 = [x_1, y]$.  Using \eqref{70eqn3} and comparing powers of $y$, $z_1$ and $z_2$, we get the following system of equations:
\begin{eqnarray}
\label{70eqn7}&&\alpha_1\beta_2 - \alpha_2 \beta_1 = 0\\
\label{70eqn8}&&\beta_1\alpha_2(\alpha_2-1)/2 - \beta_2\alpha_1(\alpha_1-1)/2 + \alpha_1\gamma_2-\alpha_2\gamma_1 = 1\\
\label{70eqn9}&&\alpha_2\beta_1(\beta_1-1)/2 - \alpha_1\beta_2(\beta_2-1)/2 - \alpha_1\beta_1\beta_2 + \alpha_2\beta_1\beta_2 + \beta_1\gamma_2 - \beta_2\gamma_1 = 0
\end{eqnarray}

We claim that  $\beta_1 = 0$. Suppose for a moment that our cliam is true.  If $\alpha_1 \neq 0$,  then it follows from \eqref{70eqn7} that $\beta_2 = 0$. Substituting $\beta_1 = 0 = \beta_2$ in \eqref{70eqn8} gives 
\begin{equation}\label{70eqn10}
\alpha_1 \gamma_2 - \alpha_2 \gamma_1 = 1.
\end{equation}
 Notice that $\beta_1 = 0$ and $\beta_2 = 0$ satisfy \eqref{70eqn9}. So if we take any  element $g_1 = x_1^{\alpha_1}y^{\gamma_1}$ with $\alpha_1 \neq 0$, then there exists $g_2 \in G$ such that $[g_1, g_2] = z_1$, where $g_2 = x_1^{\alpha_2}y^{\gamma_2}$ in which  $0 \le \alpha_2, \gamma_2 \le p-1$ satisfy \eqref{70eqn10}. If $\alpha_1 = 0$, then $\gamma_1 \neq 0$ as $g_1$ is a non-trivial element. Substituting $\alpha_1 = \beta_1 = 0$ in \eqref{70eqn8} and \eqref{70eqn9}, we respectively get
\[\gamma_1 \alpha_2 = -1 \quad\quad~~~~~~~~~~\text{and} \quad\quad~~~~~~~~~~~\gamma_1 \beta_2 = 0.\]

Since $\gamma_1 \neq 0$, we have $\beta_2 = 0$.   So if we take any  element $g_1 = y^{\gamma_1}$ with $\gamma_1 \neq 0$, then there exists $g_2 \in G$ such that $[g_1, g_2] = z_1$, where $g_2 = x_1^{\alpha_2}$ in which  $0 \le \alpha_2 \le p-1$ satisfies $\gamma_1 \alpha_2 = -1$. Hence it follows that for any non-central element $g_1 = x_1^{\alpha_1}y^{\gamma_1}z$, where $z \in \Z(G)$, there exists an element $g_2 \in G$ such that $[g_1, g_2] = z_1$.
Thus $|\{g \in G \mid z_1 \in [g, G]\}| = p^4 - p^2$. Hence by Lemma \ref{4lemma1}, $\Pr_{z_1}(G) = \frac{1}{|G|b(G)}(p^4 - p^2) = \frac{p^2-1}{p^5}$, which is the required value for $\Pr_{z_1}(G)$.

Now we prove our claim. First suppose that $\alpha_1 \neq 0$. Assume contrarily that $\beta_1 \neq 0$. Then substituting the value of $\beta_2$ from \eqref{70eqn7} in \eqref{70eqn8} and \eqref{70eqn9}, and cancelling out $\gamma_2$ using the two new equations, we get $2\beta_1 = 0$. Since $\beta_1 \neq 0$ and $p$ is odd,  this is not possible. Now assume that $\alpha_1 = 0$ and $\beta_1 \neq 0$. It follows from \eqref{70eqn7} that $\alpha_2 = 0$. Substituting $\alpha_1 = \alpha_2 =0$, in \eqref{70eqn8}, we get $0 = 1$, which is absurd. Hence $\beta_1 = 0$, which proves our claim as well as the lemma. \hfill $\Box$

\end{proof}

The following result follows from the preceding two lemmas.
\begin{thm}
Let $G$ be the group as defined in \eqref{70eqn1}. Then $|P(G)| > 1$.
\end{thm}

Now we calculate $\Pr_g(G)$ for finite Camina $p$-groups of class $3$.

\begin{prop}\label{7prop1}
Let $G$ be a finite Camina $p$-group of class $3$. Then 
\[
\Pr_g(G) = \begin{cases} \frac{1}{|G|}(\frac{|G - \gamma_2(G)|}{|\gamma_2(G)|} + \frac{|\gamma_2(G) - \gamma_3(G)|}{|\gamma_3(G)|})  \text{ if }  g \in \gamma_3(G)\\
\Pr_g(G) = \frac{1}{|G|} \frac{|G - \gamma_2(G)|}{|\gamma_2(G)|}  \text{ if }  g \in \gamma_2(G) - \gamma_3(G).
\end{cases}
\]
Moreover, $P(G) = 2$.
\end{prop}
\begin{proof}
Notice that $K(G) = \gamma_2(G)$ for a finite Camina $p$-group $G$. Also if $G$ is a finite Camina $p$-group of class $3$, then  for $g \in \gamma_2(G) - \gamma_3(G)$, $\{x \in G \mid g \in [x, G]\} = G -\gamma_2(G)$, for $g \in \gamma_3(G)$, $\{x \in G \mid g \in [x, G]\} = G -\gamma_3(G)$ and
\[|x^G| = \begin{cases} |\gamma_3(G)| = |\Z(G)| \text{ if } x \in \gamma_2(G) -\gamma_3(G)\\
|\gamma_2(G)| \text{ if } x \in G -\gamma_2(G).
\end{cases}
\]

If $1 \ne g \in \gamma_3(G)$, then by Lemma \ref{4lemma1} and the above information we have  
\[{\Pr}_g(G) = \frac{1}{|G|}{\sum_{g \in [x, G]}}\frac{1}{|x^G|} = \frac{1}{|G|}(\frac{|G - \gamma_2(G)|}{|\gamma_2(G)|} + \frac{|\gamma_2(G) - \gamma_3(G)|}{|\gamma_3(G)|} ).\]
If  $g \in \gamma_2(G) - \gamma_3(G)$, then again by Lemma \ref{4lemma1} and the above information we have 
\[\Pr_g(G) = \frac{1}{|G|}{\sum_{g \in [x, G]}}\frac{1}{|x^G|} = \frac{1}{|G|} \frac{|G - \gamma_2(G)|}{|\gamma_2(G)|}.\]
Obviously   $P(G) = 2$, since ${\Pr}_g(G)$ takes only two values for all non-trivial $g \in \gamma_2(G)$. This completes the proof. \hfill $\Box$

\end{proof}

Now we discuss some bounds on ${\Pr}_g(G)$. As a simple consequence of Lemma \ref{4lemma1}, we get the following result, which also gives a relationship between ${\Pr}_g(G)$ and ${\Pr}(G)$.
\begin{lemma}
Let $G$ be a finite group.  Then for $1 \ne g \in K(G)$, we have
\begin{equation}\label{7eqn1}
 \frac{1}{|G|b(G)}|\{x \in G \mid g \in [x,  G]\}| \le {\Pr}_g(G) \le \frac{1}{|G|}{\sum_{x \in G-\Z(G)}}\frac{1}{|x^G|}.
\end{equation} 
Moreover,

(i) equality holds on the right side if and only if $g \in [x, G]$ for all $x \in G-\Z(G)$,

(ii) equality holds on the left side if and only if $|x^G| = b(G)$ for all $x \in G$ such that $g \in [x, G]$,

(iii) ${\Pr}_g(G) <  {\Pr}(G)$ for all $1 \ne g \in \gamma_2(G)$.
\end{lemma}
\begin{proof}
Let $1 \ne g \in K(G)$. It easily follows from Lemma \ref{4lemma1} that
\[
 {\Pr}_g(G) = \frac{1}{|G|}{\sum_{g \in [x,  G]}}\frac{1}{|x^G|} \le \frac{1}{|G|}{\sum_{x \in G-\Z(G)}}\frac{1}{|x^G|}.
\]
By comparing summation terms, notice that equality holds on the right side of \eqref{7eqn1} if and only if $g \in [x, G]$ for all $x \in G-\Z(G)$. So (i) holds true. Again using Lemma \ref{4lemma1}, we have
\[
 {\Pr}_g(G) = \frac{1}{|G|}{\sum_{g \in [x,  G]}}\frac{1}{|x^G|} \ge \frac{1}{|G|}{\sum_{g \in [x,  G]}}\frac{1}{|b^G|} = \frac{1}{|G|b(G)}|\{x \in G \mid g \in [x,  G]\}|.
\]
It is again obvious to see that (ii) holds true.

Finally 
\[\frac{1}{|G|}{\sum_{x \in G-\Z(G)}}\frac{1}{|x^G|} < \frac{1}{|G|}\left({|\Z(G)| + \sum_{x \in G-\Z(G)}}\frac{1}{|x^G|} \right) = \frac{1}{|G|}{\sum_{x \in G}}\frac{1}{|x^G|} = {\Pr}(G).\]
Hence, for $1 \ne g \in \gamma_2(G)$, ${\Pr}_g(G) <  {\Pr}(G)$. \hfill $\Box$

\end{proof}

Notice that in the preceding lemma equality hold simultaneously on both sides of \eqref{7eqn1} if there exists an element $1 \ne g' \in K(G)$ such that $g' \in [x, G]$ for all $x \in G - \Z(G)$ and $|x^G| = b(G)$ for all such $x$. For such an element $g'$, ${\Pr}_{g'}(G) =  \frac{1}{|G|b(G)}|G - \Z(G)| = \frac{1}{|b(G)|}(1 - \frac{1}{|G:\Z(G)|})$. Let us set $B(G) = \frac{1}{|b(G)|}(1 - \frac{1}{|G:\Z(G)|})$.  Perhaps nothing special can be said about the relationship between $\Pr_{g}(G)$ and $B(G)$ for various  $1 \neq g \in K(G)$. This relationship highly depends on the given group $G$ as well as on the non-trivial element $g \in K(G)$.  For example, if we consider a finite Camina $p$-group of nilpotency class $3$, then it follows from Proposition \ref{7prop1} that $\Pr_g(G) > B(G)$ if $g \in \gamma_3(G)$ and $\Pr_g(G) < B(G)$ if $g \in \gamma_2(G) -\gamma_3(G)$.  
 It is not difficult to show that $\Pr_g(G) < B(G)$, for each $1 \ne g \in K(G)$, for the group $G$ defined in \eqref{70eqn1}.  
Also if  $\Pr_g(G) = B(G)$ for some group $G$ and some $1 \neq g \in K(G)$, we do not know what interesting can be said for the group $G$ itself, except the fact that such groups will have only two different conjugacy class sizes. But if,  for some group $G$, $\Pr_g(G) = B(G)$ for all $1 \neq g \in K(G)$, then we have the following nice characterization of such group. 

\begin{thm}
Let $G$ be a finite group. Then $\Pr_g(G) = B(G)$ for all $1 \neq g \in K(G)$ if and only if $G$ is isoclinic to a finite Camina special $p$-group for some prime integer $p$.
\end{thm}
\begin{proof}
Suppose that $\Pr_g(G) = B(G)$ for all $1 \neq g \in K(G)$. Then it follows from the preceding lemma that $|x^G| = b(G)$ for all $x \in G - \Z(G)$. Thus $G$ has only two different conjugacy class sizes. Now it follows from Theorem \ref{3thm2} that $G$ is isoclinic to a finite $p$-group $H$ (say) for some prime integer $p$, and the  nilpotency class of $H$ is either $2$ or $3$.  So let us work with $H$ now. Notice that $B(G) = B(H)$. Thus $\Pr_h(H) = B(H)$ for all $1 \neq h \in K(H)$ and therefore $|P(H)| = 1$. Let $1 \neq h \in K(H)$. Then $h \in [y, H]$ for all $y \in H - \Z(H)$, showing that $h \in \cap_{y \in H-\Z(H)}[y, H]$. Since this is true for each $1 \neq h \in K(H)$, we have $K(H) \in \cap_{y \in H-\Z(H)}[y, H]$. This is possible only when $[y, H] = K(H)$ for all $y \in H - \Z(H)$, since $[y, H] \subseteq K(H)$ for all $y \in H$.  We know that the nilpotency class of $H$ is either $2$ or $3$. We claim that $H$ is of class $2$. Assume that the nilpotency class of $H$ is $3$. Then there exists an element $u \in \gamma_2(G) - \Z(G)$. By what we have, it follows that $K(H) = [u, H]$. Thus $\gamma_2(H) = \gen{K(H)} \le \gamma_3(H)$, since $[u, H] \subseteq \gamma_3(H)$. This contradiction proves our claim that the nilpotency class of $H$ is $2$. Notice that in a finite $p$-group $X$ of class $2$, for any $x \in X$, $[x, X]$ is a subgroup of $\gamma_2(X)$. Since $K(H) = [y, H]$ for any element $y \in H - \\Z(H)$ and $[y, H]$ is a subgroup of $\gamma_2(H)$, it follows that  $\gamma_2(H) = \gen{K(H)} = [y, H] \le \gamma_2(H)$. Hence $[y, H] = \gamma_2(H)$ for all $y \in H - \Z(H)$.   It  now follows from Proposition \ref{3prop1}, the definition of Camina groups and Theorem \ref{3thm3} that $H$ is isoclinic to a Camina special $p$-group.

If part of the theorem follows from Theorem A. \hfill $\Box$

\end{proof}

We conclude this section with the following simple minded application of commuting probability. The following result follows from \cite[Proposition 1]{ND10}, which is proved using degree equation. Since the proof is as nice as eating grapes, we did not make efforts to re-produce a character free proof. 

\begin{prop}\label{5prop1}
Let $G$ be a non-abelian finite group and $p$ be the smallest prime dividing $|G|$. Then ${\Pr(G)} > \frac{1}{p}$ if and only if $G$ is isoclinic to an extraspecial finite $p$-group.
\end{prop}

Notice that  any extraspecial $p$-group  has only two different class sizes, namely, $1$ and $p$. Thus it follows that if $G$ is any group which is isoclinic to an  extraspecial $p$-group, then  $G$ has only two different class sizes $1$ and $p$. The following result shows that the converse also holds true. 

\begin{thm}\label{5thm1}
Let $G$ be a finite group  which has only two conjugacy class sizes $1$ and $p$, where $p$ is any prime integer. Then $G$ is isoclinic to an extraspecial finite $p$-group.
\end{thm}
\begin{proof}
Let $G$ be a finite group having only two conjugacy class sizes $1$ and $p$, where $p$ is any prime integer. Then by Theorem \ref{3thm2}, $G$ is isoclinic to a finite $p$-group $H$ (say), and therefore  $b(H) = b(G) = p$. Now it follows from Proposition \ref{4prop1} that ${\Pr}(G) =  {\Pr}(H) > 1/p$. Thus by Proposition \ref{5prop1}, $H$ is isoclinic to a finite  extraspecial $p$-group. This completes the proof of the theorem. \hfill $\Box$

\end{proof}

\begin{remark}
As mentioned in the introduction and as observed  by  Ishikawa \cite{kI99}, Theorem \ref{5thm1} can be easily obtained by using the following deep result of Vaughan-Lee \cite[Main Theorem]{vL76}, which states that for any finite $p$-group $G$, $|\gamma_2(G)| \le p^{\frac{b(G)(b(G)+1)}{2}}$.
\end{remark}

\noindent{\bf Acknowledgements.} The authors thank Pradeep K. Rai for pointing out a gap in the proof of Lemma \ref{6lemma1} in an old version of the paper.

\end{document}